\documentclass[12pt,reqno,twoside]{amsart}

\usepackage{amssymb,amsmath,amscd,enumerate,verbatim,extarrows}
\usepackage[colorlinks=true,linkcolor=blue,citecolor=blue]{hyperref}
\usepackage[latin1]{inputenc}
\usepackage{color}

\usepackage[all]{xy}
\usepackage{pstricks, pst-3d}
\newcommand{\mm}{\mathfrak m}
\newcommand{\nn}{\mathfrak n}

\newcommand{\qq}{\mathfrak q}

\newcommand{\Z}{\mathbb{Z}}

\newcommand{\Fc}{\mathcal{F}}
\newcommand{\Gc}{\mathcal{G}}

\DeclareMathOperator{\pnt}{\raise 0.5mm \hbox{\large\bf.}}

\DeclareMathOperator{\Coker}{Coker}

\DeclareMathOperator{\Img}{Im}

\DeclareMathOperator{\gr}{gr}

\DeclareMathOperator{\Tor}{Tor}

\DeclareMathOperator{\lind}{ld}
\DeclareMathOperator{\glind}{gl\,ld}
\DeclareMathOperator{\linp}{lin}
\DeclareMathOperator{\Ker}{Ker}

\DeclareMathOperator{\reg}{reg}

\DeclareMathOperator{\projdim}{pd}

\DeclareMathOperator{\Span}{Span}

\def\+#1{\relax\ifmmode\if\noexpand #1\relax \mathop{\kern
    0pt^+{#1}}\nolimits\else \kern 0pt^+\!#1 \fi\else$^*$#1\fi}

\let\:=\colon

\newtheorem{thm}{\bf Theorem}[section]
\newtheorem{lem}[thm]{\bf Lemma}
\newtheorem{cor}[thm]{\bf Corollary}
\newtheorem{prop}[thm]{\bf Proposition}

\theoremstyle{definition}
\newtheorem{defn}[thm]{\bf Definition}

\theoremstyle{plain}
\newtheorem*{thm*}{Theorem}
\newtheorem*{lem*}{Lemma}
\newtheorem*{cor*}{Corollary}
\newtheorem*{claim*}{Claim}
\newtheorem*{defn*}{Definition}

\theoremstyle{remark}
\newtheorem{rem}[thm]{Remark}

\newtheorem{ex}[thm]{Example}

\numberwithin{equation}{section}
\title[Linearity defect and applications]{Notes on the linearity defect and applications}
\author{Hop D. Nguyen}
\address{Institut f\"ur Mathematik, Friedrich-Schiller-Universit\"at Jena\\ Ernst-Abbe-Platz 2, 07743 Jena}
\address{Dipartimento di Matematica, Universit\`a di Genova, Via Dodecaneso 35, 16146 Genoa, Italy}
\email{ngdhop@gmail.com}
\thanks{The author was partially supported by the CARIGE foundation.}

\begin{document}

\begin{abstract}
The linearity defect, introduced by Herzog and Iyengar, is a numerical measure for the complexity of minimal free resolutions. Employing a characterization of the linearity defect due to \c{S}ega, we study the behavior of linearity defect along short exact sequences. We point out two classes of short exact sequences involving Koszul modules, along which linearity defect behaves nicely. We also generalize the notion of Koszul filtrations from the graded case to the local setting. Among the applications, we prove that if $R\to S$ is a surjection of noetherian local rings such that $S$ is a Koszul $R$-module, and $N$ is a finitely generated $S$-module, then the linearity defect of $N$ as an $R$-module is the same as its linearity defect as an $S$-module. In particular, we confirm that specializations of absolutely Koszul algebras are again absolutely Koszul, answering positively a question due to Conca, Iyengar, Nguyen and R\"omer. 
\end{abstract}

\maketitle

\section{Introduction}
The linearity defect, introduced by Herzog and Iyengar \cite{HIy}, measures how 
far a module is from having a linear free resolution. The notion was inspired by 
work of Eisenbud, Fl\o ystad and Schreyer \cite{EFS} on free resolutions over 
the exterior algebra. Let us recall what this invariant is. Throughout, we will only work 
with a noetherian local ring $(R,\mm,k)$ with the unique maximal ideal $\mm$ and 
the residue field $k=R/\mm$, but with appropriate changes what 
we say will also cover the graded situation where $(R,\mm,k)$ is a standard graded $k$-algebra with the graded maximal ideal $\mm$. Sometimes, we omit $k$ and write simply $(R,\mm)$. 
Let $M$ denote a finitely generated $R$-module. Let the minimal free resolution 
of $M$ over $R$ be
\[
F: \cdots \longrightarrow F_i \xlongrightarrow{\partial} F_{i-1} \xlongrightarrow{\partial} \cdots \longrightarrow F_1\xlongrightarrow{\partial} F_0 \longrightarrow 0.
\]
By definition, the differential maps $F_i$ into $\mm F_{i-1}$. Then $F$ has a filtration $\Gc^{\pnt}F$ given by $(\Gc^nF)_i=\mm^{n-i}F_i$ for all $n,i$ (where $\mm^j=R$ if $j\le 0$), for which the map 
\[
(\Gc^nF)_i = \mm^{n-i}F_i \longrightarrow (\Gc^nF)_{i-1}=\mm^{n-i+1}F_{i-1}
\]
is induced by the differential $\partial$. The associated graded complex induced by the filtration $\Gc^{\pnt} F$, denoted by $\linp^R F$, is called the {\em linear part} of $F$. We define the {\em linearity defect} of $M$ as the number
\[
\lind_R M =\sup\{i:H_i(\linp^R F) \neq 0\}.
\]
By convention, the trivial module is set to have linearity defect $0$. We say that $M$ is a {\em Koszul module} if $\lind_R M=0$. Furthermore, $R$ is called a {\it Koszul ring} if $\lind_R k=0$. In the graded case, $R$ is a Koszul algebra (i.e. $k$ has a linear free resolution as an $R$-module) if and only if $R$ is a Koszul ring, or equivalently, if and only if $\lind_R k<\infty$ \cite{HIy}. This is reminiscent of the result due to Avramov-Eisenbud and Avramov-Peeva \cite{AE}, \cite{AP} saying that $R$ is a Koszul algebra if and only if $k$ has finite Castelnuovo-Mumford regularity $\reg_R k$. It is not clear whether the analogous statement for local rings, that $\lind_R k<\infty$ implies $R$ is Koszul, holds true; see \cite{AhR}, \cite{Se} for the recent progress on this question, and \cite{Ah}, \cite{EFS}, \cite{OY}, \cite{Ro}, \cite{Y1}, \cite{Y2} for some other directions of study. For recent surveys related to free resolutions and Koszul algebras, we refer to \cite{CDR} and \cite{PS}.

The linearity defect has some connections with the other invariants coming from minimal free resolutions. It is clear from the definition that $\lind_R M\le \projdim_R M$, where $\projdim_R M$ denotes the projective dimension of $M$. Moreover, in the graded case, if $\lind_R M$ is finite, then so is the Castelnuovo-Mumford regularity $\reg_R M$ (\cite[Proposition 1.12]{HIy}).  Nevertheless, compared with the projective dimension or the regularity, the linearity defect behaves much worse along short exact sequences. 

One of the main purposes of this paper is to analyze the behavior of linearity 
defect along short exact sequences. In commutative algebra, one 
usually uses short exact sequences to bound or compute numerical invariants of 
ideals and modules. Except for the componentwise linear modules (in the sense of Herzog and Hibi \cite{HH}) which have linearity 
defect $0$, not much is known about modules with larger linearity defect, even if the base ring is a polynomial ring. Looking from these perspectives, we hope that the main theorems of this paper (Proposition \ref{lind_exact_seq}, Theorems \ref{pure_extension} and \ref{small_inclusion}) would be useful for future research on such modules. 

\c{S}ega \cite[Theorem 2.2]{Se} proved the following characterization of the linearity defect. Denoting by $\tau_s$ the canonical surjection $R/\mm^{s+1} \longrightarrow R/\mm^s$ for each $s\ge 0$, then 
\[
\lind_R M=\inf \left \{t:  ~ \begin{aligned} &\textnormal{the map} ~ \Tor^R_i(R/\mm^{s+1},M)\xlongrightarrow{\Tor^R_i(\tau_s,M)} \Tor^R_i(R/\mm^s,M) ~\textnormal{is zero} \\
&\textnormal{for all $i>t$ and all $s\ge 0$}
\end{aligned} \right \}.
\] 
Using \c{S}ega's theorem, in Section \ref{sect_lind_exact_seq}, we establish general bounds on linearity defects of modules in a short exact sequence. The main technical result of the section as well as of this paper is Proposition \ref{lind_exact_seq}. The bounds in Proposition \ref{lind_exact_seq} involve correcting terms that might appear unnatural at first sight, but they are not dispensable (see Example \ref{ex_correcting_terms}). 

In Section \ref{sect_Koszul_modules}, we describe two kinds of short exact sequence involving Koszul modules along which the linearity defect behaves well (Theorems \ref{pure_extension} and \ref{small_inclusion}). The main results 
of Section \ref{sect_Koszul_modules} will be employed to study specializations of absolutely Koszul
rings (Corollary \ref{cor_aK}), modules with linear quotients (Proposition \ref{linear_quotients}), and intersection of three linear ideals (Theorem \ref{thm_3ideals}).

An efficient method to establish Koszulness of graded algebras is constructing 
Koszul filtrations \cite{CTV}; see also, e.g., \cite{Bl}, \cite{CRV}, 
\cite{HHR}. In Section \ref{sect_Koszul_filtrations}, we generalize this method from the graded case to the local setting.

Section \ref{sect_applications} is devoted to applications of the main technical results. In the first part of this section, we prove the following (at least to us) unexpected result.

\medskip

\noindent{\bf Theorem \ref{thm_ld0}.} {\em Let $(R,\mm)\to (S,\nn)$ be a 
surjection of 
local rings such that $\lind_R S=0$. Then for any finitely generated $S$-module 
$N$, there is an equality $\lind_R N=\lind_S N$.}

\medskip

Following \cite{IyR}, $R$ is said to be {\em absolutely Koszul} if every finitely generated $R$-module has a finite linearity defect. For instance, if $Q$ is a complete intersection of quadrics and $Q\to R$ 
is a Golod surjective map of graded $k$-algebras (i.e.~ either $Q=R$ or $\reg_Q R=1$), 
then $R$ is absolutely Koszul (see \cite[Proposition 5.8, Theorem 5.9]{HIy}). 
The reader may consult \cite{CINR}, \cite{IyR} for more examples and questions 
concerning absolutely Koszul rings. As a corollary of Theorem \ref{thm_ld0}, we 
show that absolutely Koszul algebras are stable under specialization. This answers in the 
positive a question raised in \cite[Remark 3.10]{CINR}. 

In the second part of Section \ref{sect_applications}, we introduce a local version of modules with linear 
quotients \cite{HT}, and prove that it enjoys the same property as in the 
graded case. This is a simple application of Theorem \ref{pure_extension} 
(strictly speaking, we only need a special case proved in \cite[Proposition 5.3]{GM}). In contrast to the belief 
expressed in \cite[Page 461, line 6-8]{AIS} that the filtration method
neither ``covers the local situation, nor gives information on the homological 
properties of finite $R$-modules other than $k$'', we recover (partly) the results from \cite{AIS} using filtration arguments (see Proposition \ref{Conca_gen}). On the other hand, the method of 
\cite{AIS} does give stronger statements 
and the reader is encouraged to consult that paper.

In the last part of Section \ref{sect_applications}, we give another application of the main theorems of Section \ref{sect_Koszul_modules}. We prove that any intersection of three linear ideals has linearity defect zero (Theorem \ref{thm_3ideals}). Note that Francisco and Van Tuyl \cite[Theorem 4.3]{FV} prove a similar statement but their method only works for monomial ideals and does not cover our situation. We hope to show in future work how the theory of linearity defect may yield interesting information on componentwise linear ideals, e.g. via recovering the result of Francisco and Van Tuyl.

\section{General bounds}
\label{sect_lind_exact_seq}
\subsection*{Notation and background}
Let $(S,\nn)$ be a standard graded algebra over a field $k$. Let $N$ be a finitely generated graded $S$-module. The {\it Castelnuovo-Mumford regularity} of $N$ over $S$ is 
$$
\reg_S N=\max\{j-i:\Tor^S_i(k,N)_j\neq 0\}.
$$
We say that $N$ has a {\it linear resolution} over $S$ if there exists some integer $d$ such that $\Tor^S_i(k,N)_j=0$ for all $i, j$ such that $j-i\neq d$. In that case, clearly $\reg_S N=d$, and we also say that $N$ has {\it $d$-linear resolution} over $S$.

We say that $S$ is a Koszul algebra, if $k=S/\nn$ has $0$-linear resolution over $S$. The standard graded polynomial ring $k[x_1,\ldots,x_n]$ (where $n\ge 1$) is a Koszul algebra: $k$ is resolved by the Koszul complex, which is a linear resolution.

Let $M$ be a finitely generated (graded) $R$-module, where $(R,\mm)$ is our local ring (or standard graded $k$-algebra). The associated graded module of $M$ with respect to the $\mm$-adic filtration is 
$$
\gr_{\mm}M=\bigoplus_{i=0}^{\infty}\frac{\mm^iM}{\mm^{i+1}M}.
$$ 
It is a graded module over the associated graded ring $\gr_{\mm}R$, with generators in degree $0$. Recall that Koszul modules are related to linear free resolutions by the following result; we refer the reader to \cite[Theorem 2.5]{MZ} and \cite[Proposition 1.5]{HIy}.

\begin{prop}
\label{linear_res}
Let $M\neq 0$ be a finitely generated $R$-module. The following are equivalent:
\begin{enumerate}[\quad\rm(i)]
\item $M$ is a Koszul $R$-module, i.e., $\lind_R M=0$;
\item The graded $\gr_{\mm}R$-module $\gr_{\mm}M$ has $0$-linear free resolution.
\end{enumerate}
\end{prop}
\begin{defn}
We say that $R$ is a Koszul ring if the residue field $k=R/\mm$ is a Koszul module.
\end{defn}
For example, any regular local ring is Koszul, since $\gr_{\mm}R$ is isomorphic to a standard graded polynomial ring over $k$.

For the convenience of our arguments, sometimes we work with the invariant 
$$\glind R=\sup\{\lind_R M:~ \text{$M$ is a finitely generated (graded) $R$-module}\},$$ 
which is called the {\it global linearity defect} of $R$.
\begin{lem}[Conca, Iyengar, Nguyen and R\"omer, {\cite[Corollary 6.4]{CINR}}]
\label{lem_hypersurfaces}
Let $f\neq 0$ be a quadratic form in the polynomial ring $k[x_1,\ldots,x_n]$ \textup{(}where $n\ge 1$\textup{)}. Then $\glind (k[x_1,\ldots,x_n]/(f))=n-1$.
\end{lem}
For more detailed discussions of the theory of free resolutions, we refer to Avramov's monograph 
\cite{Avr} and the book of Peeva \cite{P}.

\subsection*{Bounding the linearity defect}
The starting point for our investigation is the following result due to \c{S}ega. It was stated for the local case but taking advantage of the grading, the proof works equally well in the graded case.

\begin{thm}[\c{S}ega, {\cite[Theorem 2.2]{Se}}]
\label{Sega's_theorem}
For any non-trivial finitely generated $R$-module $M$, the following are equivalent:
\begin{enumerate}[\quad\rm(i)]
\item $\lind_R M \le t$;
\item For all $i>t$ and all $s\ge 0$, the natural morphism $\Tor^R_i(R/\mm^{s+1},M)\longrightarrow \Tor^R_i(R/\mm^s,M)$ induced by the canonical surjection $R/\mm^{s+1} \to R/\mm^s$ is zero.
\end{enumerate}
\end{thm}
The main result of this section is
\begin{prop}
\label{lind_exact_seq}
Let $0\longrightarrow M\xlongrightarrow{\phi} P \xlongrightarrow{\lambda} N \longrightarrow 0$ be a short exact sequence of non-trivial finitely generated $R$-modules. Define the \textup{(}possibly infinite\textup{)} numbers:
\begin{align*}
d_M &=\inf\{m\ge 0: \textnormal{the connecting map $\Tor^R_{i+1}(k,N) \longrightarrow \Tor^R_i(k,M)$ is zero for all $i\ge m$}\},\\
d_P &=\inf\{m\ge 0: \textnormal{the natural map $\Tor^R_i(k,M) \xlongrightarrow{\Tor^R_i(k,\phi)} \Tor^R_i(k,P)$ is zero for all $i\ge m$} \},\\
d_N &=\inf\{m\ge 0: \textnormal{the natural map $\Tor^R_i(k,P) \xlongrightarrow{\Tor^R_i(k,\lambda)} \Tor^R_i(k,N)$ is zero for all $i\ge m$}\}.
\end{align*}
Then there are inequalities
\begin{enumerate}[\quad\rm(i)]
\item $\lind_R N \le \max\{\min\{d_P,d_M+1\},\lind_R P, \lind_R M +1\},$ 
\item $\lind_R P \le \max\{\min\{d_M,d_N\},\lind_R M, \lind_R N\},$
\item $\lind_R M \le \max\{\min\{d_N-1,d_P\},\lind_R N-1,\lind_R P\}$.
\end{enumerate}
\end{prop}
Several comments are in order.
\begin{rem}
\label{rem_d_numbers}

(i) In general, we have the following inequalities:
\begin{align*}
d_M &\le \min \{\projdim_R M+1, \projdim_R N\},\\
 d_P &\le \min \{\projdim_R M +1,\projdim_R P+1\},\\
d_N &\le \min \{\projdim_R P+1, \projdim_R N+1\}.
\end{align*}
Hence if $P$ is a free module, then $d_P, d_N\le 1$. Similar things happen if $M$ or $N$ is a free module. 
\medskip

(ii) Since $\Tor^R_i(k,M) \longrightarrow \Tor^R_i(k,P) \longrightarrow \Tor^R_i(k,N) \longrightarrow \Tor^R_{i-1}(k,M)$ is an exact sequence for all $i$, we also have other interpretations for the numbers $d_M, d_N, d_P$. For example, 
\[
d_M=\inf \{m\ge 0: \textnormal{the map $\Tor^R_i(k,M) \xlongrightarrow{\Tor^R_i(k,\phi)} \Tor^R_i(k,P)$ is injective for all $i\ge m$}\}.
\]
Therefore, the two numbers $d_P$ and $d_M$ are not simultaneously finite unless $\projdim_R M<\infty$. Similar statements hold for the pairs $d_P$ and $d_N$, $d_M$ and $d_N$.

\medskip

(iii) The above interpretation of $d_M$ indicates that the first inequality of \eqref{lind_exact_seq} relates $\lind_R N$ with asymptotic properties of the map $\Tor^R_i(k,M)\xlongrightarrow{\Tor^R_i(k,\phi)} \Tor^R_i(k,P)$. Similar comments apply to the inequalities for linearity defects of $M$ and $P$.
\end{rem}
\begin{ex}
In general, none of the numbers $d_M,d_N,d_P$ is finite, even if $R$ is Koszul and $M,N,P$ are Koszul modules. For example, take $R=k[x,y]/(xy)$. Consider the exact sequence with natural maps
$$
0\longrightarrow (x^3,y^2)\xlongrightarrow{\phi} (x^2,y^2)\xlongrightarrow{\lambda} \frac{(x^2,y^2)}{(x^3,y^2)}\longrightarrow 0.
$$
The (2-periodic) minimal free resolution of $k$ over $R$ is given by
$$
F: \cdots \longrightarrow R^2\xrightarrow{\left(\begin{matrix}
y & 0 \\
0 & x 
\end{matrix}\right)}R^2 \xrightarrow{\left(\begin{matrix}
x & 0 \\
0 & y 
\end{matrix}\right)} R^2 
\xrightarrow{\left(\begin{matrix}
y & 0 \\
0 & x
\end{matrix}\right)}R^2 \xrightarrow{\left(\begin{matrix}
x & y
\end{matrix}\right)} R \longrightarrow 0.
$$
Let $P=(x^2,y^2)$, we want to compute $\Tor^R_i(k,P)$. Note that $P\otimes_R F_i=P\oplus P$ for $i\ge 1$. Fix $i\ge 2$, the map $P\otimes_R F_{2i}\longrightarrow P\otimes_R F_{2i-1}$ is given by 
$$
(a,b)\mapsto (ya,xb)
$$ and the map $P\otimes_R F_{2i+1}\longrightarrow P\otimes_R F_{2i}$ is given by 
$$
(u,v)\mapsto (xu,yv).
$$
Let $\partial$ be the differential of $P\otimes_R F$, then 
\begin{align*}
\Ker \partial_{2i}&=(x^2)\oplus (y^2),\\
\Img \partial_{2i+1}&=(x^3)\oplus (y^3).
\end{align*}
Therefore $\Tor^R_{2i}(k,P)\cong ((x^2)/(x^3))\oplus ((y^2)/(y^3))$. Similarly, setting $M=(x^3,y^2)$, then it holds that $\Tor^R_{2i}(k,M)\cong ((x^3)/(x^4))\oplus ((y^2)/(y^3))$. In particular, 
\begin{align*}
\Ker \left(\Tor^R_{2i}(k,\phi) \right)&=((x^3)/(x^4))\\
\Img \left(\Tor^R_{2i}(k,\phi) \right)&=((y^2)/(y^3)).
\end{align*}
This implies that $d_P=\infty$. Denote $N=P/M$, then from the exact sequence of $\Tor$, we also infer that $d_M=d_N=\infty$. Note that $N\cong R/(x,y)=k$, so $\lind_R N=0$. One can check that $M, P$ are Koszul modules: By Lemma \ref{lem_hypersurfaces}, $\lind_R R/U\le 1$ for any ideal $U\subseteq \mm$. Hence $\lind_R U=0$.
\end{ex}

\medskip

Now we are going to prove Proposition \ref{lind_exact_seq}. First we have several simple but very useful observations. 
\begin{lem}
\label{rigidity}
Let $M \xlongrightarrow{\phi} P$ be an $R$-linear map between finitely generated $R$-modules.
\begin{enumerate}[\quad\rm(i)]
\item If for some $\ell \ge \lind_R M+1$, the map $\Tor^R_{\ell -1}(k,M)\xlongrightarrow{\Tor^R_{\ell -1}(k,\phi)} \Tor^R_{\ell-1}(k,P)$ is injective, then the map
\[
\Tor^R_i(R/\mm^s, M)\xlongrightarrow{\Tor^R_i(R/\mm^s,\phi)} \Tor^R_i(R/\mm^s, P)
\] 
is injective for all $i\ge \ell$ and all $s\ge 0$.
\item If for some $\ell \ge \lind_R P+1$, the map $\Tor^R_{\ell -1}(k,M)\xlongrightarrow{\Tor^R_{\ell -1}(k,\phi)} \Tor^R_{\ell-1}(k,P)$ is zero, then the map 
\[
\Tor^R_i(R/\mm^s, M)\xlongrightarrow{\Tor^R_i(R/\mm^s,\phi)} \Tor^R_i(R/\mm^s, P)
\] 
is zero for all $i\ge \ell$ and all $s\ge 0$.
\end{enumerate}
\end{lem}
\begin{proof}
Consider the following commutative diagram, where $\rho, \psi_{i-1}$ are induced by $\phi$, and $\alpha^i_M, \alpha^i_P$ are connecting maps:
\begin{displaymath}
\xymatrix{\Tor^R_i(R/\mm^s,M) \ar[d]^{\rho} \ar[r]^{\alpha^i_M} & \Tor^R_{i-1}(\mm^s/\mm^{s+1},M)  \ar[d]^{\psi_{i-1}}\\
\Tor^R_i(R/\mm^s,P) \ar[r]^{\alpha^i_P}      & \Tor^R_{i-1}(\mm^s/\mm^{s+1},P).
}
\end{displaymath}

(i) By induction on $i$ and using the above diagram for $s=1$, we see that $\Tor^R_i(k,M)\xlongrightarrow{\Tor^R_i(k,\phi)} \Tor^R_i(k,P)$ is injective for all $i\ge \ell-1$. Note that as $i\ge \lind_R M+1$, by Theorem \ref{Sega's_theorem}, the map $\alpha^i_M$ is injective. Next let $s\ge 0$ be arbitrary, again using the diagram and the fact that $\mm^s/\mm^{s+1}$ is either $0$ (equivalently, $\mm^s=0$) or isomorphic to a direct sum of copies of $k$, we deduce that $\rho=\Tor^R_i(R/\mm^s,\phi)$ is also injective.

\medskip

(ii) Similarly, by induction on $i$ and using the diagram for $s=1$, $\Tor^R_i(k,M)\xlongrightarrow{\Tor^R_i(k,\phi)} \Tor^R_i(k,P)$ is the zero map for all $i\ge \ell-1$. Note that since $i\ge \lind_R P+1$, $\alpha^i_P$ is injective. Then for arbitrary $s\ge 0$, using the diagram, we see that $\rho=\Tor^R_i(R/\mm^s,\phi)$ is the zero map as well.
\end{proof}
\begin{proof}[Proof of Proposition \ref{lind_exact_seq}]
Below, we omit the superscript $R$ in the notation of Tor modules for simplicity.

(i) For the proof of the inequality $\lind_R N\le \max\{d_P,\lind_R P,\lind_R M+1\}$, we may assume that $\ell=\max\{d_P,\lind_R P, \lind_R M+1\}<\infty$. For each $i>\ell, s\ge 0$, from the exact sequence
\[
0\longrightarrow \mm^s/\mm^{s+1} \longrightarrow R/\mm^{s+1} \longrightarrow R/\mm^s \longrightarrow 0,
\]
we get the following commutative diagram with exact rows and columns
\begin{displaymath}
\xymatrix{
   & \Tor_i(R/\mm^s,M) \ar[d]  \ar[r]^{\alpha^i_M} & \Tor_{i-1}(\mm^s/\mm^{s+1},M) \ar[d]^{\psi_{i-1}} \ar[r]& \Tor_{i-1}(R/\mm^{s+1},M) \ar[d]^{\kappa}\\
0 \ar[r] & \Tor_i(R/\mm^s,P) \ar[d]^{\pi} \ar[r]^{\alpha^i_P} & \Tor_{i-1}(\mm^s/\mm^{s+1},P)\ar[d]^{\pi}\ar[r] & \Tor_{i-1}(R/\mm^{s+1},P)\\
         & \Tor_i(R/\mm^s,N) \ar[d]^{\gamma} \ar[r]^{\alpha^i_N} & \Tor_{i-1}(\mm^s/\mm^{s+1},N)\ar[d]^{\gamma}&\\ 
0 \ar[r] & \Tor_{i-1}(R/\mm^s,M)  \ar[r]^{\alpha^{i-1}_M}        & \Tor_{i-2}(\mm^s/\mm^{s+1},M)               &}          
\end{displaymath}
By \c{S}ega's Theorem \ref{Sega's_theorem} and the fact that $i\ge \max\{\lind_R P+1, \lind_R M+2\}$, we have $\alpha^i_P,\alpha^{i-1}_M$ are injective. Note that $\mm^s/\mm^{s+1}$ is either zero if $\mm^s=0$ or otherwise a direct sum of copies of $k$, therefore by hypothesis, we have $\psi_{i-1}=0$. Now we need to show that $\alpha^i_N$ is also injective. This is a simple diagram chasing. Hence $\lind_R N\le \ell$.

Next we want to show that $\lind_R N\le \max\{d_M+1,\lind_R P,\lind_R M+1\}$. We lose nothing by assuming that the right-hand side is finite. Take $i\ge \max\{d_M+1,\lind_R P,\lind_R M+1\}+1$. Look at the exact sequence $\Tor_{i-1}(k,N)\longrightarrow \Tor_{i-2}(k,M) \xlongrightarrow{\Tor_{i-2}(k,\phi)} \Tor_{i-2}(k,P)$. Since $i-2\ge d_M$, the first map is zero. Hence the second map is injective. Now $i-1\ge \lind_R M+1$, hence by Lemma \ref{rigidity}(i), $\kappa$ is injective. Therefore by diagram chasing, again $\lind_R N < i$.

\medskip

(ii), (iii): The proofs are similar to part (i).
\end{proof}

We give various instances to show that none of the inequalities of Proposition \ref{lind_exact_seq} is true without the correcting terms $d_M, d_N$ and $d_P$. In fact, we will exhibit examples of exact sequences $0\to M \to P \to N \to 0$ where one of the modules has infinite linearity defect and the other two have small linearity defect.

\begin{ex}
\label{ex_correcting_terms} 
Let $R=k[x,y,z,t]/((x,y)^2+(z,t)^2)$, $\mm$ its graded maximal ideal. Observe that $\mm^3=0$. By result of Roos \cite[Theorem 2.4]{Roos}, there exists a graded $R$-module with infinite linearity defect. Explicitly, by \cite[Formula (5.2)]{Roos} and \cite[Proposition 1.8]{HIy}, the cokernel of the map $R(-1)^3\longrightarrow R^2$ given by the matrix 
\[
\left(\begin{matrix}
y & x+3t & t \\
z & -t & x+t
\end{matrix}\right)
\]
is such a module. Let $F=R(-1)^3, G=R^2,M=\Ker(F\to G)$ and $N=\Img(F\to G)$. Note that $F$ is the projective cover of $N$. Since $N\subseteq \mm G$, we have $\mm^2N=0$ (recall that $\mm^3=0$). Clearly $\lind_R N=\lind_R M=\infty$.

\medskip

(i) The $R$-module $N$ is an extension of Koszul $R$-modules. Indeed, we have an exact sequence
\[
0\longrightarrow \mm N \longrightarrow N \longrightarrow N/\mm N \longrightarrow 0.
\]
Now $\mm N$ and $N/\mm N$ are both annihilated by $\mm$, so they are Koszul modules.  So there is an extension of Koszul $R$-modules which has infinite linearity defect.

\medskip

(ii) Since $M\subseteq \mm F$, we also have an exact sequence
\[
0\longrightarrow M \longrightarrow \mm F \longrightarrow \mm N \longrightarrow 0.
\]
Now $\mm F$ is a Koszul module and $\mm N$ is also Koszul as noted above. So the kernel of a surjection of Koszul modules may have infinite linearity defect.

\medskip

(iii) Now $N$ is an $(R/\mm^2)$-module so we can take the beginning of the minimal graded $(R/\mm^2)$-free resolution of $N$, say (without grading notation)
\[
0\longrightarrow D \longrightarrow (R/\mm^2)^r \longrightarrow N \longrightarrow 0.
\]
So $D$ is annihilated by $\mm$, hence $D$ is a Koszul $R$-module. Also $\lind_R (R/\mm^2)=1$ but $\lind_R N=\infty$.

We do not know if there exists a short exact sequence in which the first two modules are Koszul but the cokernel has infinite linearity defect.
\end{ex}
We record a few consequences of Proposition \ref{lind_exact_seq}. Interestingly, 
we can extract information about the linearity defect from any (minimal or not) free resolution of a module: If $P_{\pnt}$ is a free resolution of $N$, then 
$\lind_R N=r\ge 1$ if and only if $r$ is the minimal number $i$ such that $\Omega_i(N)=\Img(P_i\longrightarrow P_{i-1})$ 
is Koszul. If $N$ is a Koszul 
module then so is $\Omega_i(N)$ for every $i\ge 1$.

\begin{cor}
\label{cor_freemod}
Let $0\longrightarrow M \xlongrightarrow{\phi} P \xlongrightarrow{\lambda} N \longrightarrow 0$ be an exact sequence of non-trivial finitely generated $R$-modules. Then 
\begin{enumerate}[\quad\rm(i)]
\item $\lind_R N \le \min \{\max\{\projdim_R P+1,\lind_R M+1\}, \max\{\lind_R P, \projdim_R M+1\} \}$,
\item $\lind_R P \le \min \{\max\{\projdim_R M+1,\lind_R N\}, \max\{\lind_R M, \projdim_R N\}\}$,
\item $\lind_R M \le \min \{\max\{\lind_R N-1, \projdim_R P\},\max\{\projdim_R N, \lind_R P\} \}$.
\end{enumerate}
In particular, we have:
\begin{enumerate}[\quad\rm(a)]
\item If $P$ is free, then $\lind_R M=\lind_R N-1$ if $\lind_R N\ge 1$ and $\lind_R M=0$ otherwise.
\item If one of the modules has finite projective dimension, then the other two have both finite or both infinite linearity defects.
\end{enumerate}
\end{cor}
\begin{proof}
For (i): using Proposition \ref{lind_exact_seq}, we get
$$
\lind_R N\le \max\{d_P,\lind_R P,\lind_R M+1\}.
$$
Since $d_P\le \max\{\projdim_R P+1,\projdim_R M+1\}$ by Remark \ref{rem_d_numbers}(i), and $\lind_R M\le \projdim_R M$, the desired inequalities follow. Similar arguments work for (ii) and (iii).

For (a): since $\projdim_R P=0$, from (i) and (iii), we get the inequalities
\begin{align*}
\lind_R N &\le \lind_R M+1,\\
\lind_R M &\le \max\{\lind_R N-1,0\}.
\end{align*}
This yields the conclusion of (a). The remaining assertion is a consequence of (i)--(iii).
\end{proof}
\section{Short exact sequences involving Koszul modules}
\label{sect_Koszul_modules}
We describe quite concretely the behavior of linearity defect for some short exact sequences involving Koszul modules without any assumption on the ground ring. Firstly, using results in Section \ref{sect_lind_exact_seq}, we can control the linearity defect for certain ``pure" extensions of a Koszul module. The first main result of this section is as follows.

\begin{thm}
\label{pure_extension}
Let $0\longrightarrow M'\xlongrightarrow{\phi'} P' \xlongrightarrow{\lambda'} N' \longrightarrow 0$ be a short exact sequence of non-zero finitely generated $R$-modules where
\begin{enumerate}[\quad\rm(i)]
\item $M'$ is a Koszul module;
\item $M'\cap \mm P'=\mm M'$.
\end{enumerate}
Then there are inequalities $\lind_R P'\le \lind_R N' \le \max\{\lind_R P',1\}$. In particular, $\lind_R N'=\lind_R P'$ if $\lind_R P'\ge 1$ and $\lind_R N'\le 1$ if $\lind_R P'=0$.

Moreover \textup{(}see Green and Mart\'inez-Villa \cite[Propositions 5.2 and 5.3]{GM}\textup{)}, $\lind_R N'=0$ if and only if $P'$ is a Koszul module and $M'\cap \mm^s P'=\mm^s M'$ for all $s\ge 1$ .
\end{thm}
\begin{proof}
We will show that $d_{M'}=0$, or equivalently, $\Tor^R_i(k,M')\xlongrightarrow{\Tor^R_i(k,\phi')} \Tor^R_i(k,P')$ is injective for each $i\ge 0$.

This is clear for $i=0$ thanks to the equality $M'\cap \mm P'=\mm M'.$ Now using 
Lemma \ref{rigidity}(i) where $\lind_R M'=0,\ell=1$, we get the desired claim.

Next, using Proposition \ref{lind_exact_seq} where $d_{M'}=0$ and $\lind_R 
M'=0$, we obtain that
\[
\lind_R N' \le \max\{1,\lind_R P'\},
\]
and that
\[
\lind_R P'  \le \lind_R N'.
\]
The first part of the result is already proved. Next we give a new proof for the result of Green and Mart\'inez-Villa.

\medskip

Now assume that $P'$ is a Koszul module and $M'\cap \mm^sP'=\mm^s M'$ for all $s\ge 1$. We show that $\lind_R N'=0$. Consider the diagram with obvious connecting and induced maps
\begin{displaymath}
\xymatrix{
0\ar[r] & \Tor_1(R/\mm^s,M') \ar[d]  \ar[r]^{\alpha_{M'}}         & \Tor_0(\mm^s/\mm^{s+1},M') \ar[d]^{\psi} \ar[r]& \Tor_0(R/\mm^{s+1},M') \ar[d]^{\kappa}\\
0\ar[r] & \Tor_1(R/\mm^s,P') \ar[d]^{\pi} \ar[r]^{\alpha_{P'}}    & \Tor_0(\mm^s/\mm^{s+1},P')\ar[d] \ar[r] & \Tor_0(R/\mm^{s+1},P') \\
& \Tor_1(R/\mm^s,N') \ar[d]^{\gamma} \ar[r]^{\alpha_{N'}} & \Tor_0(\mm^s/\mm^{s+1},N')\ar[d]        & \\ 
& \Tor_0(R/\mm^s,M')  \ar[r]                              & 0                                       & }          
\end{displaymath}
We know that $\lind_R N'\le 1$ by the preceding part, so by Theorem \ref{Sega's_theorem}, it is enough to show that $\Tor_1(R/\mm^s, N')\longrightarrow \Tor_0(\mm^s/\mm^{s+1},N')$ is injective for all $s\ge 1$. Clearly $\Img \gamma =(M'\cap \mm^s P')/\mm^s M'=0$, so $\pi$ is surjective. According to the hypothesis, $\psi$ is injective. By the snake lemma, $\Ker \alpha_{P'}\longrightarrow \Ker \alpha_{N'} \longrightarrow \Coker \alpha_{M'} \longrightarrow \Coker \alpha_{P'}$ is exact. But $\Ker \alpha_{P'}=0=\Ker \kappa$, hence  $\Ker \alpha_{N'}=0$.

\medskip

Finally, assume that $\lind_R N'=0$, then by the first part, $\lind_R P'\le \lind_R N'=0$. Assume that on the contrary, $M'/\mm^{s+1} M'\longrightarrow P'/\mm^{s+1}P'$ is not injective for some $s\ge 1$. Choose $s$ minimal with this property, we will show that $\lind_R N'\ge 1$. Again in the above diagram, $\Img \gamma=0$ by the choice of $s$. Using the snake lemma, we get $\Ker \alpha_{N'}\cong \Ker \kappa \neq 0$. Therefore $\lind_R N'\ge 1$, a contradiction. The proof of the theorem is completed.
\end{proof}

\begin{rem}

(i) The conclusion of the theorem is not true if $M'$ is not a Koszul module or $M'\cap \mm P'\neq \mm M'$. Firstly, consider the exact sequence
$$
0\longrightarrow (x^2,y^2)\longrightarrow (x^2,y^2,xz)\longrightarrow \frac{(x^2,y^2,xz)}{(x^2,y^2)}\longrightarrow 0
$$ 
over $R=k[x,y,z]$. Set $M'=(x^2,y^2)$, $P'=(x^2,y^2,xz)$ and $N'=(x^2,y^2,xz)/(x^2,y^2)$. Then $N'\cong R/(x)$, so $\lind_R N'=0$.  It is clear that $M'\cap \mm P'=\mm M'$, $M'$ is not Koszul, and $\lind_R P'=1>\lind_R N'$.

Secondly, consider the exact sequence
$$
0\longrightarrow D\longrightarrow (R/\mm^2)^r\longrightarrow N\longrightarrow 0
$$
in Example \ref{ex_correcting_terms}(iii). Note that $D$ is Koszul, and $D\subseteq \mm (R/\mm^2)^r$, hence the condition (ii) of Theorem \ref{pure_extension} is not satisfied. In this case, we also have $\lind_R N=\infty >\max\{1,\lind_R (R/\mm^2)^r\}=1$.

\medskip

(ii) In the situation of Theorem \ref{pure_extension}, it may happen that $\lind_R P'=0$ but $\lind_R N'=1$. Consider the exact sequence of $(R=)~k[x,y]$-modules
\[
0\longrightarrow (x^2) \longrightarrow (x^2,y) \longrightarrow (x^2,y)/(x^2)\longrightarrow 0.
\]
Clearly $\lind_R (x^2)=\lind_R (x^2,y)=0$, while $N'=(x^2,y)/(x^2) \cong R/(x^2)$, so $\lind_R N'=1$. 
\end{rem}

\medskip

\begin{rem}
The fact that $\Tor^{R'}_i(k,M')\xlongrightarrow{\Tor^R_i(k,\phi')} \Tor^{R'}_i(k,P')$ is always injective for all $i\ge 0$ was shown by Mart\'inez-Villa and Zacharia \cite[Proposition 3.2]{MZ} by different means. Note that therein, it is not necessary to assume that $R$ is a Koszul ring. A similar remark applies when comparing Corollary \ref{reg_pure_extension} below with \cite[Corollary 3.3]{MZ}.
\end{rem}
We also obtain interesting information about behavior of projective dimension and regularity for sequences satisfying the hypothesis of Theorem \ref{pure_extension}.
\begin{cor}[See {\cite[Corollary 3.3]{MZ}}]
\label{reg_pure_extension}
 With the hypotheses of Theorem \ref{pure_extension}, there is an equality 
$$
\projdim_R P' = \max\{\projdim_R M',\projdim_R N'\}.
$$ 

If $R$ is a standard graded algebra and $M', P',N'$ are finitely generated graded modules, then
\[
\reg_R P'=\max\{\reg_R M',\reg_R N'\}.
\]
\end{cor}
\begin{proof}
For each $i\ge 0$, we have an exact sequence
\[
0\longrightarrow \Tor^R_i(k,M') \longrightarrow \Tor^R_i(k,P') \longrightarrow \Tor^R_i(k,N')\longrightarrow 0.
\]
This clearly implies our desired equalities.
\end{proof}

We also have the control over linearity defect for ``small inclusion" in a Koszul module. The next result demonstrates that if $N$ is any finitely generated $R$-module and $P$ is any Koszul module which surjects onto $N$ in such a way that $M=\Ker(P\to N) \subseteq \mm P$, the module $M$ behaves as if it was the first syzygy module of $N$. See Corollary \ref{reg_small_inclusion} for another result of this type. 

\begin{thm}
\label{small_inclusion}
Let $0\longrightarrow M\xlongrightarrow{\phi} P \xlongrightarrow{\lambda} N \longrightarrow 0$ be a short exact sequence of non-zero finitely generated $R$-modules where
\begin{enumerate}[\quad\rm(i)]
\item $P$ is a Koszul module;
\item $M\subseteq \mm P$.
\end{enumerate}
Then there are inequalities $\lind_R N-1 \le \lind_R M \le \max\{0,\lind_R N-1\}$. In particular, $\lind_R N=\lind_R M+1$ if $\lind_R M\ge 1$ and $\lind_R N\le 1$ if $\lind_R M=0$. 

Furthermore, $\lind_R N=0$ if and only if $M$ is a Koszul module and $M\cap \mm^{s+1}P=\mm^s M$ for all $s\ge 0$.
\end{thm}
\begin{rem}
The conclusion of the above result is false in general if $P$ is not Koszul or $M\not\subseteq \mm P$. 

(i) Firstly, look at the sequence
$$0\longrightarrow D\longrightarrow (R/\mm^2)^r\longrightarrow N\longrightarrow 0$$
in Example \ref{ex_correcting_terms}(iii). It is easy to verify that $D\subseteq \mm (R/\mm^2)^r$ but $\lind_R (R/\mm^2)^r=1$, and $\lind_R D=0$ while $\lind_R N=\infty>\lind_R D+1=1$.

\medskip

(ii) Secondly, look at the sequence 
$$
0\longrightarrow M\longrightarrow \mm F\longrightarrow \mm N\longrightarrow 0
$$
in Example \ref{ex_correcting_terms}(ii). We know that $\mm F$ is Koszul, but $M\nsubseteq \mm^2F$. Indeed, otherwise $\mm M=0$ and thus $M$ would be Koszul, while in fact $\lind_R M=\infty$. We also know that $\max\{0,\lind_R (\mm N)-1\}=0<\lind_R M=\infty$.
\end{rem}
\begin{proof}[Proof of Theorem \ref{small_inclusion}]

For the first part: Observe that $d_P=0$, i.e., $\Tor^R_i(k,M)\xlongrightarrow{\Tor^R_i(k,\phi)} \Tor^R_i(k,P)$ is the zero 
map for each $i\ge 0$. Indeed, this follows Lemma \ref{rigidity}(ii) 
since $\lind_R P=0$ and $\Tor^R_0(k,M)\longrightarrow \Tor^R_0(k,P)$ is the zero map.

Now using Proposition \ref{lind_exact_seq} where $d_P=0$ and the fact that $P$ is Koszul, we see that
\[
\lind_R N\le \max\{0,0, \lind_R M+1\}=\lind_R M+1,
\]
and
\[
\lind_R M \le \max\{0,\lind_R N-1,0\}=\max\{0,\lind_R N-1\}.
\]
This gives the first part of the result.

\medskip

For the second part: first assume that $M$ is a Koszul module and $M\cap \mm^{s+1}P=\mm^s M$ for all $s\ge 0$.  Since $M\subseteq \mm P$, there is an exact sequence
\[
0\longrightarrow M \longrightarrow \mm P \longrightarrow \mm N \longrightarrow 0.
\]
We show that the induced sequence of graded $\gr_{\mm}R$-modules
\begin{equation}
\label{assoc_seq}
0\longrightarrow (\gr_{\mm}M)(-1) \longrightarrow \gr_{\mm}P \longrightarrow \gr_{\mm}N\longrightarrow 0
\end{equation}
is exact. Indeed, since $M\subseteq \mm P$, we have $0\longrightarrow P/\mm P \longrightarrow N/\mm N \longrightarrow 0$ is exact.
For each $s\ge 1$, we prove that the sequence below is exact
\[
0\longrightarrow \frac{\mm^{s-1}M}{\mm^sM} \longrightarrow \frac{\mm^sP}{\mm^{s+1}P} \xlongrightarrow{\overline{\lambda}} \frac{\mm^sN}{\mm^{s+1}N} \longrightarrow 0.
\]
Let $\bar{x}\in \Ker \overline{\lambda}$ where $x\in \mm^s P$. Then $\lambda(x)\in \mm^{s+1}N$, and as $\lambda$ is surjective, we see that $\lambda(x-y)=0$ for some $y\in \mm^{s+1}P$. This implies that $x-y \in M \cap \mm^s P =\mm^{s-1}M$; the last equality holds by the hypothesis. Now $y\in \mm^{s+1}P$, therefore
\[
\bar{x}\in \frac{\mm^{s-1}M}{\mm^sM},
\]
as desired. The exactness on the left follows from the equality $M\cap \mm^{s+1}P=\mm^s M$. So the sequence \eqref{assoc_seq} is exact.

Denote $A=\gr_{\mm}R$. Now the first two modules in \eqref{assoc_seq} have linear $A$-free resolutions, moreover $\reg_A (\gr_{\mm}M)(-1)=1$ and $\reg_A \gr_{\mm}P=0$. Therefore $\gr_{\mm}N$ also has $0$-linear $A$-free resolution. So $N$ is a Koszul $R$-module by Proposition \ref{linear_res}.

\medskip

Conversely, assume that $\lind_R N=0$. From the first part, we already know that $M$ must be Koszul. 

Since $M\subseteq \mm P$, we have the following commutative diagram in which the rows are exact and the vertical maps are natural inclusions
\begin{displaymath}
\xymatrix{
0  \ar[r] & M \ar[r] \ar[d]^{=} & \mm P \ar[r] \ar[d] & \mm N \ar[r] \ar[d] & 0\\
0  \ar[r] & M \ar[r] & P \ar[r] &  N \ar[r] & 0
}
\end{displaymath}
This induces the following commutative diagram of homology for each $s\ge 0$
\begin{displaymath}
\xymatrix{
\Tor_1(R/\mm^s,\mm N) \ar[d]^{\alpha}  \ar[r]^{\beta} & \Tor_0(R/\mm^s,M) \ar[r]^{\gamma} \ar[d]^{=} & \Tor_0(R/\mm^s,\mm P) \ar[d] \\
\Tor_1(R/\mm^s,N)   \ar[r] & \Tor_0(R/\mm^s,M) \ar[r] & \Tor_0(R/\mm^s,P) 
}
\end{displaymath}
Thanks to the fact that $N$ is Koszul and Lemma \ref{rigidity}, $\alpha$ is the zero map. Hence from the commutativity of the left square, we get that $\beta$ is also the zero map. In particular, $\Ker \gamma=0$, which is equivalent to the fact that $M\cap \mm^{s+1}P=\mm^sM$ for all $s\ge 0$. The proof of the theorem is completed.
\end{proof}

\begin{cor}
\label{reg_small_inclusion}
With the hypothesis of Theorem \ref{small_inclusion}, there is an equality 
$$
\projdim_R N =\max\{\projdim_R M+1,\projdim_R P\}.
$$ 

If $R$ is a standard graded algebra and $M, P, N$ are finitely generated graded modules then
\[
\reg_R N =\max\{\reg_R M-1,\reg_R P\}.
\]
\end{cor}
\begin{proof}
As noted in the proof of Theorem \ref{small_inclusion}, for each $i\ge 0$, the map $\Tor^R_i(k,M)\xlongrightarrow{\Tor^R_i(k,\phi)}\Tor^R_i(k,P)$ is trivial. Hence for each such $i$, we have a short exact sequence
\[
0\longrightarrow \Tor^R_i(k,P) \longrightarrow \Tor^R_i(k,N) \longrightarrow \Tor^R_{i-1}(k,M)\longrightarrow 0.
\]
This desired conclusion follows.
\end{proof}

We also recover the following result of Green and Mart\'inez-Villa  \cite[Proposition 5.5]{GM}.
\begin{cor}
\label{max_ideal}
Let $R$ be a Koszul local ring. Let $M\neq 0$ be a Koszul $R$-module. Then $\mm^i M$ is also a Koszul module for all $i\ge 1$.
\end{cor}
\begin{proof}
It is enough to consider the case $i=1$. Look at the exact sequence
\[
0\longrightarrow \mm M \longrightarrow M \longrightarrow M/\mm M \longrightarrow 0.
\]
Note that $M/\mm M$ is an $R/\mm$-module, so as $R$ is a Koszul ring, $\lind_R M/\mm M=0$. Using the first part of Theorem \ref{small_inclusion}, we get 
$\lind_R (\mm M)=0$ as well.
\end{proof}

\section{Koszul filtrations}
\label{sect_Koszul_filtrations} 

In the graded setting, the notion of Koszul filtration in \cite{CTV} has proved to be useful to detect Koszul property of algebras. We extend this notion to the local setting in the present section.
\begin{defn}
\label{defn_filtr}
Let $(R,\mm,k)$ be a local ring. Let $\Fc$ be a collection of ideals. We say that $\Fc$ is a {\it Koszul filtration of} $R$ if the following simultaneously hold:
\begin{enumerate}
\item[(F1)] $(0),\mm \in \Fc$,
\item[(F2)] for every ideal $I\in \Fc$ and all $s\ge 1$, we have $I\cap \mm^{s+1}=\mm^s I$,
\item[(F3)] for every ideal $I\neq (0)$ of $\Fc$, there exist a finite filtration $(0)=I_0 \subset I_1 \subset \cdots \subset I_n=I$ and elements $x_j\in \mm$, such that for each $j=1,\ldots,n$, $I_j\in \Fc$, $I_j=I_{j-1}+(x_j)$ and $I_{j-1}:x_j\in \Fc$.
\end{enumerate}
\begin{rem}

(i) It is straightforward to check that the usual notion of Koszul filtration for standard graded algebras satisfies the conditions of Definition \ref{defn_filtr}. 

\medskip

(ii) Condition (F3) in our definition of Koszul filtration is more involved than the corresponding condition in \cite[Definition 1.1]{CTV}; the reason behind is to make the induction process in the proof of Theorem \ref{Koszul_filtration} below to work. In the case of graded Koszul filtrations, the condition is automatically satisfied.
\end{rem}
\end{defn}

The following theorem extends a well-known result about algebras with Koszul 
filtration \cite{CTV}. 

\begin{thm}
\label{Koszul_filtration}
Let $(R,\mm,k)$ be a local ring with a Koszul filtration $\Fc$. Then:
\begin{enumerate}[\quad \rm(i)]
\item For any ideal $I\in \Fc$, $R/I$ is a Koszul $R$-module.
\item $R$ is a Koszul ring.
\item $R/I$ is a Koszul ring for any $I\in \Fc$.
\end{enumerate}  
\end{thm}
\begin{proof}

(i) We may assume that $\mm\neq (0)$, otherwise $R$ is a field and $\Fc=\{(0)\}$. We prove by induction on $i\ge 1$ that for every ideal $I\in \Fc$ and for every $s\ge 0$, the map
\[
\Tor^R_i(R/\mm^s,R/I)\longrightarrow \Tor^R_{i-1}(\mm^s/\mm^{s+1},R/I)
\] 
is injective.

Firstly, assume that either $i=1$. Since $\Tor^R_1(R/\mm^s,R/I)=(I\cap \mm^s)/\mm^sI$, the natural map 
\[
\Tor^R_1(R/\mm^{s+1},R/I)\longrightarrow \Tor^R_1(R/\mm^s,R/I)
\]
is zero by condition (F2) for Koszul filtrations. Hence the connecting map is injective. 

Now assume that $i\ge 2$ and the desired statement already holds up to $i-1$. It is harmless to assume that $I\neq (0)$. By condition (F2) for Koszul filtrations, there exist a finite filtration $(0)=I_0 \subset I_1 \subset \cdots \subset I_n=I$ and elements $x_j\in \mm$ for $j=1,\ldots,n$ such that for all $1\le j \le n$, $I_j\in \Fc, I_j=I_{j-1}+(x_j)$ and $I_{j-1}:x_j\in \Fc$. To our purpose, it suffices to prove by induction on $j$ that for every $0\le j\le n$ and for every $s\ge 0$, the map 
\[
\Tor^R_i(R/\mm^s,R/I_j)\longrightarrow \Tor^R_{i-1}(\mm^s/\mm^{s+1},R/I_j)
\] 
is injective.

Indeed, this is true if $j=0$ since $I_0=(0)$. Assume that $1\le j\le n$ and the statement is true up to $j-1$.

Denote $x=x_j, L=I_{j-1}$ so that $I_j=L+(x)$. We have an exact sequence
\[
0\longrightarrow R/(L:x) \xlongrightarrow{\cdot x} R/L\longrightarrow R/I_j \longrightarrow 0.
\]
We have $\mm/\mm^2\cong k^t$ for some $t\ge 1$. Consider the commutative diagram with obvious connecting and induced maps
\begin{displaymath}
\xymatrix{\Tor^R_{\ell}(R/\mm,R/(L:x)) \ar[r]^{\rho_{\ell}} \ar[d] & \Tor^R_{\ell}(R/\mm,R/L) \ar[d]^{\tau_{\ell}}\\
\Tor^R_{\ell-1}(\mm/\mm^2,R/(L:x))  \ar[r]^{\rho^t_{\ell-1}}     & \Tor^R_{\ell-1}(\mm/\mm^2,R/L).
}
\end{displaymath}
We prove by induction on $\ell$ that $\rho_{\ell}$ is the zero map for all $0\le \ell \le i$. Indeed, the case $\ell=0$ follows since $R/(L:x) \subseteq x(R/L)$. Assume that $1\le \ell\le i$ and $\rho_j$ is the trivial map for all $j\le \ell-1$. Observe that $\tau_{\ell}$ is injective: if $\ell < i$ then this follows from the induction on $i$, while if $\ell=i$ then, recalling that $L=I_{j-1}$, this follows from the induction on $j$. Since $\rho_{\ell-1}$ is the zero map, from the diagram, so is $\rho_{\ell}$. This finishes the induction on $\ell$.

Now consider the diagram with obvious connecting and induced maps
\begin{displaymath}
\xymatrix{
         &                                                      & \Tor_{i-1}(\mm^s/\mm^{s+1},R/(L:x)) \ar[d]^{\rho_{i-1}}\\
0 \ar[r] & \Tor_i(R/\mm^s,R/L) \ar[d] \ar[r]^{\alpha^i_2} & \Tor_{i-1}(\mm^s/\mm^{s+1},R/L)\ar[d] \\
         & \Tor_i(R/\mm^s,R/I_j) \ar[d] \ar[r]^{\alpha^i_3} & \Tor_{i-1}(\mm^s/\mm^{s+1},R/I_j)\ar[d]\\ 
0 \ar[r] & \Tor_{i-1}(R/\mm^s,R/(L:x))  \ar[r]^{\alpha^{i-1}_1}        & \Tor_{i-2}(\mm^s/\mm^{s+1},R/(L:x)) }          
\end{displaymath}
By the hypothesis of the induction on $j$ (respectively, on $i$), the map $\alpha^i_2$ (resp.~ $\alpha^{i-1}_1$) are injective. We know from the previous paragraph that $\rho_{i-1}$ is the zero map. Hence by a snake lemma argument, $\alpha^i_3$ is also injective. This finishes the induction on $j$, and also the proof of part (i).

\medskip

(ii) From (i), taking $I=\mm$, we get that $\lind_R k=0$. This shows that $R$ is Koszul.

\medskip

(iii) Since $R/I$ is a Koszul $R$-module, the module $\gr_{\mm}(R/I)$ has linear resolution over $\gr_{\mm} R$. This shows that 
$$\reg_{\gr_{\mm}(R/I)}k= \reg_{\gr_{\mm}R}k=0,$$
where the first equality follows from Proposition \ref{prop_reg0}(iii), and the second from part (ii). Therefore $\gr_{\mm}(R/I)$ is a Koszul algebra, equivalently, $R/I$ is a Koszul ring.
\end{proof}

\section{Applications}
\label{sect_applications}
\subsection*{Change of rings}
Recall the following well-known change of rings statement concerning regularity (see, for example, \cite[Proposition 3.3]{CDR}).
\begin{prop}
\label{prop_reg0}
Let $R\to S$ be a surjection of standard graded $k$-algebras. Let $N$ be a finitely generated graded $S$-module. Then:
\begin{enumerate}[\quad \rm(i)]
\item It always holds that $\reg_R N\le \reg_R S+\reg_S N$.
\item If $\reg_R S\le 1$ then $\reg_S N\le \reg_R N$.
\item In particular, if $\reg_R S=0$ then $\reg_R N=\reg_S N$.
\end{enumerate}
\end{prop}
Now we deduce from Theorem \ref{small_inclusion} the following analog of 
Proposition \ref{prop_reg0}(iii). Recall from \cite{IyR} that $R$ is called {\em 
absolutely Koszul} if every finitely generated $R$-module $M$ has finite 
linearity defect.
\begin{thm}
\label{thm_ld0}
Let $(R,\mm)\to (S,\nn)$ be a surjection of local rings such that $\lind_R S=0$. Then for any finitely generated $S$-module $N$, there is an equality
$$\lind_R N=\lind_S N.$$ 
In particular, $\glind S\le \glind R$. If $R$ is absolutely Koszul then so is 
$S$.
\end{thm}
\begin{proof}
We claim that $\lind_R N=0$ if and only if $\lind_S N=0$. Denote $A=\gr_{\mm}R, B=\gr_{\mm}S, U=\gr_{\mm}N$ we get $\reg_A B=0$ by hypothesis.  Hence applying Proposition \ref{prop_reg0}, we get that $\reg_B U=\reg_A U$. The claim then follows from the last equality.
 
To prove that $\lind_R N=\lind_S N$, firstly consider the case $\lind_R N=\ell<\infty$. We prove by induction on $\ell$. The case $\ell=0$ was treated above. 

Assume that $\ell\ge 1$, then by the claim, it follows that $\lind_S N\ge 1$. Let $0\to M\to P\to N \to 0$ be the beginning of the minimal $S$-free resolution of $N$. Since $M\subseteq \mm P$ and $\lind_R P=\lind_R S=0$, we get from Theorem \ref{small_inclusion} that $\lind_R M=\ell-1$. Since $\lind_S N\ge 1$, we also have $\lind_S M=\lind_S N-1$. By induction hypothesis, $\lind_R M=\lind_S M$, thus $\lind_R N=\lind_S N$.

Now consider the case $\lind_R N=\infty$ and by way of contradiction, assume that $\lind_S N<\infty$. Again looking at the syzygy modules of $N$ as an $S$-module and using Theorem \ref{small_inclusion}, we reduce the general situation to the case $\lind_R N=\infty$ and $\lind_S N=0$. The last two equalities contradict the claim above. So in any case $\lind_R N=\lind_S N$. 

The remaining assertions are obvious.
\end{proof}
\begin{ex}
The following example shows that in Theorem \ref{thm_ld0}, one cannot weaken the hypothesis that $R\to S$ is surjective to ``$R\to S$ is a finite morphism". Take $R=k$ and $S=k[x,y]/(x^2,y^2)$. Then $S$ is a finite, free $R$-module so $\lind_R S=0$. On the other hand, by \cite[Theorem 6.7]{HIy}, $\glind S=\infty$ and $\glind R=0$.  Hence the conclusion of Theorem \ref{thm_ld0} does not hold for $R\to S$.
\end{ex}
\begin{rem}
The analog of Proposition \ref{prop_reg0}(i) for linearity defect is completely false: even if $R\to S$ is a Golod map of Koszul algebras (hence $\lind_R S=1$), it is possible for some Koszul $S$-module $N$ to have infinite linearity defect over $R$. For example, take $R=k[x,y,z,t]/((x,y)^2+(z,t)^2)$ as in Example \ref{ex_correcting_terms}. Consider the map $R\to R/\mm^2$. Since $R$ is Koszul, $R\to S$ is a Golod map. Consider the $R$-module $N$ in Example \ref{ex_correcting_terms}. Recall that $N$ is also an $S$-module, and of course $\lind_S N=0$. On the other hand, we know that $\lind_R N=\infty$.

This example also shows that the conclusion of Theorem \ref{thm_ld0} does not hold if $\lind_R S\ge 1$.
\end{rem}
\begin{rem}
\label{rem_ld1}
In view of Proposition \ref{prop_reg0}(ii), we can ask:

Let $R\to S$ be a surjection of local rings such that $\lind_R S\le 1$. Is it true that $\lind_S N\le \lind_R N$ for any finitely generated $S$-module $N$?

But the answer is {\em no}, even if $R$ and $S$ are Koszul.  Indeed, take $R=k[x,y]/(x^2)$ and $S=R/(y^2)$, then $\lind_R S=1$ and from Lemma \ref{lem_hypersurfaces}, $\glind R=1$. However as noted above, $\glind S=\infty$. Hence the question has a negative answer. If we do not insist that $S$ is Koszul, we can take $R=k[x,y]$ and $S=k[x,y]/(x^3)$. Then $\lind_R S=1$, $\lind_S k=\infty$ while $\lind_R k=0$.
\end{rem}
As a corollary to Theorem \ref{thm_ld0}, we prove that specializations of 
absolutely Koszul algebras are again absolutely Koszul. There are many open questions concerning absolutely Koszul rings; see \cite[Remark 3.10]{CINR}. By \cite[Theorem 2.11]{IyR}, if $R$ is a graded 
absolutely Koszul algebra and $x\in R_1$ an $R$-regular linear form such that 
$R/(x)$ is absolutely Koszul, then so is $R$. The converse is given by
\begin{cor}
\label{cor_aK}
Let $(R,\mm)$ be an absolutely Koszul local ring and $x\in \mm\setminus \mm^2$ be such that $\overline{x}\in \mm/\mm^2$ is $\gr_{\mm}R$-regular. Then $R/(x)$ is also absolutely Koszul.
\end{cor}
\begin{proof}
Since $\overline{x}$ is $\gr_{\mm}R$-regular, we get that $\lind_R R/(x)=0$. The result follows from Theorem \ref{thm_ld0}.
\end{proof}
\begin{ex}
Let $(R,\mm)\to (S,\nn)$ be a finite, flat morphism of local rings. One may ask whether for any finitely generated $S$-module $N$ such that $\lind_S N=0$, we also have $\lind_R N=0$? This is true if $\projdim_S N=0$: in that case $\projdim_R N=0$. But in general, this is far from the truth. For any $n\ge 1$, take $R=k[x_1,\ldots,x_n]$ and $S=k[x_1,\ldots,x_n,y_1,\ldots,y_n]/(y_1^2,\ldots,y_n^2)$. We have a surjection $S\longrightarrow k[x_1,\ldots,x_n]/(x_1^2,\ldots,x_n^2)$ given by 
\begin{align*}
x_i&\mapsto x_i,\\
y_i&\mapsto x_i,
\end{align*}
for $1\le i\le n$.
The kernel is $(x_1-y_1,\ldots,x_n-y_n)$. Since $x_1-y_1,\ldots,x_n-y_n$ is an $S$-regular sequence, we see that $\lind_S k[x_1,\ldots,x_n]/(x_1^2,\ldots,x_n^2)=0$. On the other hand, direct computations with the Koszul complex show that $\lind_R k[x_1,\ldots,x_n]/(x_1^2,\ldots,x_n^2)=n$.
\end{ex}


\subsection*{Modules with linear quotients}
Recall the following notion due to Herzog and Hibi.
\begin{defn}[Componentwise linear modules]
Let $R$ be a standard graded $k$-algebra. Let $M$ be a finitely generated graded $R$-module. Then $M$ is said to be componentwise linear if for every $d\in \Z$, the submodule $M_{\left<d\right>}=(m\in M: \deg m=d)\subseteq M$ has $d$-linear resolution as an $R$-module.
\end{defn}
R\"omer proved in his thesis \cite{Ro} the following characterization of componentwise linear modules over Koszul algebras; see, e.g., \cite[Theorem 5.6]{IyR} for a proof.
\begin{thm}[R\"omer]
\label{thm_Roemer}
Assume that $R$ is a Koszul algebra. Then for any finitely generated graded $R$-module $M$, the following are equivalent:
\begin{enumerate}[\quad \rm(i)]
\item $M$ is componentwise linear;
\item $M$ is a Koszul module over $R$.
\end{enumerate}
\end{thm} 
We will give a criterion for Koszul modules over a local ring $R$. First we introduce the following generalization of ideals with linear quotients \cite[Section 1]{HT}. The later are an ideal-theoretic analog of rings with Koszul filtrations.
\begin{defn}[Modules with linear quotients]
\label{defn_linear_quotients}
Let $M\neq 0$ be a finitely generated $R$-module with a minimal system of generators 
$m_1,\ldots,m_t$. Let $I_i=(m_1,\ldots,m_{i-1}):_R m_i$. We say that $M$ has 
linear quotients if for each $i=1,\ldots,t$, the cyclic module $R/I_i$ is a 
Koszul module.
\end{defn}
In view of R\"omer's theorem \ref{thm_Roemer}, the following result is a generalization of \cite[Theorem 3.7]{LZ}, \cite[Corollaries 2.4, 2.7]{SV}, \cite[Proposition 3.7]{M}. A notable feature is that no assumption on the ring is needed, while in the three results just cited, $R$ has to be at least a Koszul algebra.
\begin{prop}
\label{linear_quotients}
Let $M\neq 0$ be a module with linear quotients with a minimal system of generators $m_1,\ldots,m_t$ as in Definition \ref{defn_linear_quotients}. Then each of the submodule $(m_1,\ldots,m_i)$ of $M$ is a Koszul module for $1\le i \le t$. In particular, $M$ is a Koszul module.

Moreover, we have
\begin{align*}
&\beta_s(M)=\sum_{i=1}^{t}\beta_s(R/I_i) ~\textnormal{for all $s\ge 0$},\\
&\projdim_R M=\max_{1\le i\le t}\{\projdim_R (R/I_i)\}.
\end{align*}
If $R$ is a graded algebra, $M$ a graded module, $\deg m_i=d_i$ for $1\le i\le t$, then we also have
\begin{align*}
&\beta_{s,j}(M)=\sum_{i=1}^{t}\beta_{s,j-d_i}(R/I_i) ~\textnormal{for all $s,j\ge 0$},\\
&\reg_R M=\max_{1\le i\le t}\{\reg_R (R/I_i) + d_i\}.
\end{align*}
\end{prop}
\begin{proof}
Denote $M_i=(m_1,\ldots,m_i)$. Observe that $(m_i)/\left((m_i)\cap M_{i-1}\right)=(m_i)/I_im_i \cong R/I_i$ for each $1\le i\le t$. In fact, this follows since if $xm_i \in I_im_i \subseteq (m_1,\ldots,m_{i-1})$ then $x\in (m_1,\ldots,m_{i-1}):_R m_i=I_i$. Since $m_1,\ldots,m_t$ are a minimal system of generators, we have $M_{i-1}\cap \mm M_i=\mm M_{i-1}$. Therefore using induction on $i$, the short exact sequence
\[
0\to M_{i-1} \to M_i \to R/I_i \to 0,
\]
and Theorem \ref{pure_extension}, we conclude that $M_i$ is a Koszul module for every $1\le i\le t$. 

For the remaining statements, we note that from the proof of Corollary \ref{reg_pure_extension}, the induced sequence
\[
0\longrightarrow \Tor^R_s(k,M_{i-1})\longrightarrow \Tor^R_s(k,M_i) \longrightarrow \Tor^R_s(k,R/I_i) \longrightarrow 0
\]
is exact for every $i$ and every $s$. In the graded case, we use the corresponding facts for the exact sequence
\[
0\to M_{i-1} \to M_i \to (R/I_i)(-d_i) \to 0.
\]
The proof is finished.
\end{proof}

To illustrate the filtration techniques of Theorem \ref{Koszul_filtration} and Proposition 
\ref{linear_quotients}, we present a slight improvement of a result due to 
Avramov, Iyengar and \c{S}ega (which in the notation of the next result corresponds to the case $\qq$ is a 
principal ideal).
\begin{prop}[See {\cite[Theorems 1.1, 3.2]{AIS}}]
\label{Conca_gen}
Let $(R,\mm,k)$ be a local ring. Let $\qq \subseteq \mm$ be an ideal such that $\mm^2=\qq \mm$ and $\qq^2=0$. Let $y_1,\ldots,y_e$ be a minimal generating set of $\qq$ where $y_i\in \mm$. Then the collection of ideals
\[
\Fc=\{0,(y_1),(y_1,y_2),\ldots,(y_1,\ldots,y_{e-1})\} ~ \bigcup ~ \{\textnormal{$I\subseteq \mm$: $I$ contains $\qq$}\}
\]
is a Koszul filtration for $R$. Moreover, any non-trivial finitely generated $R$-module $M$ that satisfies the condition $\qq M=0$ is a Koszul module.
\end{prop}
\begin{proof}
The case $\qq=(0)$ is trivial as the reader may check, so we assume that $\qq \neq 0$. Clearly $\Fc$ contains $(0)$ and $\mm$. We begin by checking the condition (F2) for Koszul filtrations. Firstly consider the case $I\neq (0)$ is an ideal containing $\qq$. As $\mm^3=(0)$, the condition is trivial for $s\ge 2$. For $s=1$, $\mm^2\subseteq \qq \subseteq I$, hence $\mm^2\cap I=\mm^2=\mm \qq \subseteq \mm I \subseteq \mm^2 \cap I$. In particular, all containments in the last string are in fact equalities. 

Next consider the case $I=(y_1,\ldots,y_i)$ where $1\le i\le e-1$. Take $x\in I\cap \mm^2=I\cap \qq \mm$, then $x=r_1y_1+\cdots+r_iy_i=s_1y_1+\cdots+s_ey_e$ where $r_i\in R, s_i\in \mm$. Then we have $(r_1-s_1)y_1+\cdots+(r_i-s_i)y_i-s_{i+1}y_{i+1}-\cdots-s_ey_e=0$. But $y_1,\ldots,y_e$ are linearly independent modulo $\mm \qq$, therefore $r_j-s_j\in \mm$ for all $1\le j\le i$. Hence $r_j\in \mm$ for all $1\le j \le i$, and so $x\in \mm I$, as desired.
 
Now we verify condition (F3). Let $I\subseteq \mm$ be an ideal of $R$ containing $\qq$. Let $z_1,\ldots,z_n$ be an irredundant set of elements of $I$ such that $I=\qq+(z_1,\ldots,z_n)$. Define $I_0=(0), I_1=(y_1),I_2=(y_1,y_2),\ldots,I_e=(y_1,\ldots,y_e)=\qq,I_{e+1}=\qq+(z_1),\ldots,I_{e+n}=\qq+(z_1,\ldots,z_n)=I$. Observe that $I_j:y_{j+1}$ are proper ideals containing $\qq$ for $0\le j\le e$ and $I_{e+t-1}:z_t=\mm$ for  $1\le t\le n$. This argument also implies that the condition (F3) holds if $I$ is among ideals of the type $(y_1,\ldots,y_i)$ where $1\le i\le e-1$. Hence $\Fc$ is a Koszul filtration of $R$. In particular, $R/I$ is a Koszul $R$-module if $I\subseteq \mm$ is an ideal containing $\qq$.

Let $M$ be a finitely generated $R$-module with $\qq M=0$. Let $m_1,\ldots,m_t$ 
be a minimal system of generators of $M$. Immediately, we get 
$(m_1,\ldots,m_{i-1}):m_i$ is a proper ideal containing $\qq$ for each 
$i=1,\ldots,t$. Therefore by the first part of the result, $M$ has linear
quotients. In particular, $M$ is a Koszul module by Proposition 
\ref{linear_quotients}.
\end{proof}
\begin{rem}
Note that  in the previous result, if $\qq$ is a principal ideal, using the 
machinery in \cite{AIS} one obtains more information about modules over the 
local ring $R$: every finitely generated $R$-module has a Koszul syzygy module, 
i.e. $R$ is absolutely Koszul.
\end{rem}
We have checked and would like to inform the reader that there are a number of 
other results concerning Koszul rings and modules that can be proved using 
filtration arguments, for example the main results of Ahangari Maleki in \cite{Ah} 
(except those concerning regularity). To keep the exposition coherent, we decide 
to leave further details to the interested reader.


\subsection*{Intersection of three linear ideals}

In this subsection, let $R$ be a polynomial ring over $k$. We say a homogeneous ideal of 
$R$ is a {\em linear ideal} if it is generated by linear forms. In general, an 
intersection of four linear ideals is not Koszul: 
$$(xy,zt)=(x,z)\cap (x,t)\cap (y,z)\cap (y,t).$$
The main theorems of Section \ref{sect_Koszul_modules} together with a result 
Derksen-Sidman \cite{DS} give the following statement for the intersection of 
three linear ideals. 
\begin{thm}
\label{thm_3ideals}
Let $R=k[x_1,\ldots,x_n]$ be a polynomial ring \textup{(}where $n\ge 0$\textup{)}. Let $I,J,K$ be 
linear ideals of $R$. Then $I\cap J\cap K$ is a Koszul module.
\end{thm}

\begin{proof}
Denote $\mm$ the unique graded maximal ideal of $R$. Denote by $\mu(I)$ the 
minimal number of generators of $I$. We use induction on $\mu(I)+\mu(J)+\mu(K)$ 
and $n=\dim R$.
 
If one of the numbers $\mu(I),\mu(J),\mu(K)$ is zero then $I\cap 
J\cap K=(0)$. If $n=0$, then again $I\cap J\cap K=(0)$. Hence we can now 
consider the case $n\ge 1$ and $\mu(I),\mu(J),\mu(K)\ge 1$.

We claim that it is possible to reduce the general situation to the case $I\cap 
J,J\cap K, K\cap I\subseteq \mm^2$.

Firstly, if there exists a linear form $0\neq x\in I\cap J\cap K$, consider the 
exact sequence
$$
0\longrightarrow (x)\longrightarrow I\cap J\cap K \longrightarrow \frac{I\cap J\cap K}{(x)}\longrightarrow 0.
$$
Clearly $(x)\cap \mm(I\cap J\cap K)=\mm (x)$, hence using Theorem 
\ref{pure_extension}, there is an inequality
\[
\lind_R (I\cap J\cap K)\le \lind_R \frac{I\cap J\cap K}{(x)}=\lind_{R/(x)} 
\frac{I\cap J\cap K}{(x)}.
\]
The second equality is due to Theorem \ref{thm_ld0}. By induction on $\dim 
R$, 
$\lind_{R/(x)} (I\cap J\cap K)/(x)=0$. Hence the conclusion is true in this case.

Therefore it is harmless to assume that $I\cap J\cap K$ contains no linear 
forms. With this assumption, $I\cap J\cap K\subseteq \mm^2$. Consider the case 
where one of $I\cap J,I\cap K, J\cap K$ contains a linear forms, say $0\neq 
x\in 
J\cap K$ for $x\in R_1$.

Denote by $\overline{(\cdot)}$ the residue class in $R/(x)$. Look 
at the exact sequence
$$
0\longrightarrow I\cap J\cap K\longrightarrow I\longrightarrow \frac{I+J\cap K}{J\cap K}\longrightarrow 0.
$$
We have $I\cap J\cap K\subseteq \mm^2\cap I=\mm I$, hence by Theorem 
\ref{small_inclusion}, we get
$$
\lind_R (I\cap J\cap K)\le \max\{0,\lind_R \frac{I+J\cap K}{J\cap 
K}-1\}=\max\{0,\lind_{R/(x)} \frac{\overline{I+(x)}+\overline{J}\cap 
\overline{K}}{\overline{J}\cap 
\overline{K}}-1\}.
$$
The equality is due to Theorem \ref{thm_ld0}. On the other hand, arguing 
similarly for the following exact sequence in $R/(x)$
\[
0\longrightarrow \overline{I+(x)}\cap \overline{J}\cap \overline{K} \longrightarrow \overline{I+(x)} 
\longrightarrow \frac{\overline{I+(x)}+\overline{J}\cap 
\overline{K}}{\overline{J}\cap 
\overline{K}} \longrightarrow 0,
\]
we see that 
$$\lind_{R/(x)} \frac{\overline{I+(x)}+\overline{J}\cap 
\overline{K}}{\overline{J}\cap 
\overline{K}} \le \lind_{R/(x)} (\overline{I+(x)}\cap \overline{J}\cap 
\overline{K})+1=1,$$ 
with the equality following from the induction hypothesis on $\dim R$. 
Therefore $\lind_R (I\cap J\cap K)=0$ in 
this case as well. Summing up, we have reduced the general situation to the 
case when $I\cap J,I\cap K, J\cap K$ are all contained in $\mm^2$.

For any $1\le p\le n$, let $\Span(x_1,\ldots,x_p)$ be the $k$-vector subspace of $R_1$ generated by 
$x_1,\ldots,x_p$. By change of coordinates, we can assume that $I,J,K$ are minimally generated as 
follow
\begin{align*}
&I=(x_1,\ldots,x_p),\\
&J=(y_1,\ldots,y_q),\\
&K=(z_1,\ldots,z_r,a_1+b_1,\ldots,a_s+b_s),
\end{align*}
where $a_i\in \Span(x_1,\ldots,x_p),b_i\in \Span(y_1,\ldots,y_q)$, and 
$x_1,\ldots,x_p,y_1,\ldots,y_q,z_1,\ldots,z_r$ are linearly independent. 

Since $(J\cap K)_1=0$, we get that $a_1,\ldots,a_s$ linearly independent. Hence 
by change of coordinates, we can assume that $a_i=x_i$. Similarly, we can 
assume that $b_i=y_i$. Hence it remains to consider the case
\begin{align*}
&I=(x_1,\ldots,x_p),\\
&J=(y_1,\ldots,y_q),\\
&K=(z_1,\ldots,z_r,x_1+y_1,\ldots,x_s+y_s),
\end{align*}
where $s\le \min\{p,q\}$. 

This is the content of Lemma \ref{lem_special_case} below. The proof of the 
theorem is completed.
\end{proof}
The final difficulty in the proof of Theorem \ref{thm_3ideals} is resolved by
\begin{lem}
\label{lem_special_case}
Let $R=k[x_1,\ldots,x_p,y_1,\ldots,y_q,z_1,\ldots,z_r]$ be a polynomial ring 
\textup{(}where $p,q,r\ge 0$\textup{)}. Then for any $s\le \min\{p,q\}$, the ideal
$$
(x_1,\ldots,x_p)\cap (y_1,\ldots,y_q)\cap 
(x_1+y_1,\ldots,x_s+y_s,z_1,\ldots,z_r)
$$
is a Koszul module. 
\end{lem}
\begin{proof}
Denote $H=(x_1,\ldots,x_p)\cap (y_1,\ldots,y_q)\cap 
(x_1+y_1,\ldots,x_s+y_s,z_1,\ldots,z_r)$. By Theorem \ref{thm_Roemer}, we are 
left with proving that $H_{\left<c\right>}=(a\in H: \deg a=c)$ has 
$c$-linear resolution for all $c\in \Z$. 

Denote $\mm=R_+$. By \cite[Theorem 2.1]{DS}, $\reg H\le 3$, so in particular 
$H$ 
is generated in degree $2$ and $3$. The last fact implies that 
$H_{\left<c\right>}=\mm H_{\left<c-1\right>}$ for all $c\ge 4$. Hence by 
Corollary \ref{max_ideal}, it is enough to show that $H_{\left<2\right>}$ and 
$H_{\left<3\right>}$ have linear resolutions. Note that as $\reg H\le 3$,
$$
H_{\left<3\right>}=H_{\ge 3}=(m\in H: \deg m\ge 3)
$$
has linear resolution by a well-known result of Eisenbud and Goto \cite[Theorem 1.2(1)]{EG}. Hence we are 
left with $H_{\left<2\right>}.$

We will show that
$H_{\left<2\right>}$ equals $L$, the ideal of 2-minors of the following 
generic matrix
\[
\left(\begin{matrix}
x_1 & x_2 & \ldots & x_s\\
y_1 & y_2 & \ldots & y_s
\end{matrix}
\right).
\]
This implies the desired conclusion. Clearly $L\subseteq H_{\left<2\right>}$, 
since for all $1\le i<j\le s$, the following equality holds
\[
x_iy_j-x_jy_i=(x_i+y_i)y_j-(x_j+y_j)y_i\in (x_1,\ldots,x_s)\cap 
(y_1,\ldots,y_s)\cap (x_1+y_1,\ldots,x_s+y_s)\subseteq H.
\]
We will show that $H_{\left<2\right>}\subseteq L$. 

Denote $H'=(x_1,\ldots,x_p)\cap (y_1,\ldots,y_q)\cap (x_1+y_1,\ldots,x_s+y_s)$ then $H'\subseteq H$. We claim that $H'_{\left<2\right>}=H_{\left<2\right>}$. The left-hand side is clearly contained in the right-hand one. Note that $H\subseteq (x_1,\ldots,x_p)(y_1,\ldots,y_q)$ so any minimal generator $f$ of $H_{\left<2\right>}$ is a $k$-linear combination of $x_1y_1, x_1y_2,\ldots,x_1y_q,\ldots,x_py_q$. Since $f\in (x_1+y_1,\ldots,x_s+y_s,z_1,\ldots,z_r)$, a Gr\"obner basis argument using a suitable elimination order gives that $f\in H\cap (x_1+y_1,\ldots,x_s+y_s)=H'$. Hence $H'_{\left<2\right>}=H_{\left<2\right>}$.

Repeating the same argument, we see that
$$
H_{\left<2\right>}=\left((x_1,\ldots,x_s)\cap (y_1,\ldots,y_s)\cap (x_1+y_1,\ldots,x_s+y_s)\right)_{\left<2\right>}.
$$
In other words, we can assume that $p=q=s, r=0$. Equip the gradings for the variables of $R$ 
as follow: $\deg x_i=\deg y_i=(0,\ldots,0,1,\ldots,0)$, the $i$-th standard 
basis vector of $\Z^s$. Then $H, L$ are $\Z^s$-graded with respect to this 
grading, furthermore, the $\Z^s$-grading is compatible with the usual 
$\Z$-grading. 

Take $a\in H_{\left<2\right>}$ a $\Z^s$-graded element of degree $2$. Then 
taking into account the fact that 
$$H\subseteq (x_1,\ldots,x_s)\cap 
(y_1,\ldots,y_s)=(x_1,\ldots,x_s)(y_1,\ldots,y_s),$$
$a$ has the form $\alpha x_iy_j-\beta x_jy_i$ for some $1\le i<j\le s$ and 
$\alpha,\beta\in k$. As $\alpha x_iy_j-\beta x_jy_i \in 
(x_1+y_1,\ldots,x_s+y_s)$, degree considerations yield that
\[
\alpha x_iy_j-\beta x_jy_i \in (x_i+y_i,x_j+y_j) 
\] 
and by further simple calculations, we get $\alpha=\beta$. Hence 
$a=\alpha(x_iy_j-x_jy_i)\in L$, as desired.
\end{proof}

\section*{Acknowledgements}
We would like to thank Aldo Conca, Srikanth Iyengar and Tim R\"omer for some 
inspiring discussions related to the content of this paper. We are indebted to the referee for the attentive reading 
of a previous version of this paper, and more importantly, for the many thoughtful suggestions that have lead to substantial improvement in the presentation.

\end{document}